\documentclass[11pt]{amsart}
\usepackage{pdfsync}

\usepackage{amsmath, amssymb}
\usepackage{tikz}
\usepackage{hyperref}
\usetikzlibrary{decorations.pathmorphing}

\newcommand{\LF}{\text{LF}}

\DeclareMathOperator{\dist}{dist}

\newcommand{\fc}{\mathfrak c}

\newcommand{\calR}{\mathcal R}

\newcommand{\bbN}{{\mathbb N}}

\newcommand{\bbR}{\mathbb R}

\newcommand{\bbT}{\mathbb T}

\newcommand{\cZ}{{\mathcal Z}}

\newcommand{\e}{\varepsilon}

\newtheorem{thm}{Theorem}
\newtheorem{theorem}{Theorem}[section]

\newtheorem{remark}[theorem]{Remark}

\newtheorem{question}[theorem]{Question}

\newtheorem{lemma}[theorem]{Lemma}
\newtheorem{prop}[theorem]{Proposition}

\theoremstyle{definition}

\newcounter{my_enumerate_counter}
\newcommand{\pushcounter}{\setcounter{my_enumerate_counter}{\value{enumi}}}
\newcommand{\popcounter}{\setcounter{enumi}{\value{my_enumerate_counter}}}

\newcommand{\bbZ}{\mathbb Z}

\newcommand{\fb}{\mathfrak b}

\DeclareMathOperator{\Ad}{Ad}

\newcommand{\F}{\mathcal F}

\DeclareMathOperator{\Aut}{Aut}

\title[A simple $\mathrm{C}^*$-algebra  with finite nuclear dimension\dots]
{A simple $\mathrm{C}^*$-algebra  with finite nuclear dimension which is not $\cZ$-stable}

\author{Ilijas Farah}

\address{Department of Mathematics and Statistics\\
York University\\
4700 Keele Street\\
North York, Ontario\\ Canada, M3J 1P3\\
and Matematicki Institut, Kneza Mihaila 35, Belgrade, Serbia}

\urladdr{http://www.math.yorku.ca/$\sim$ifarah}

\email{ifarah@mathstat.yorku.ca}
\thanks{Dedicated to Stuart White on the occasion of his 33rd birthday}

\author{Dan Hathaway}
\address{Mathematics Department\\
University of Michigan\\
Ann Arbor, MI 48109--1043, U.S.A.}
\email{danhath@umich.edu}

\author{Takeshi Katsura}

\address{Department of Mathematics\\
Faculty of Science and Technology\\
Keio University\\
3-14-1 Hiyoshi, Kouhoku-ku, Yokohama\\
JAPAN, 223-8522}

\email{katsura@math.keio.ac.jp}

\author{Aaron Tikuisis}

\address{Mathematisches Institut\\
Westf\"alische Wilhelms-Universit\"at M\"unster\\
48149 M\"unster, Germany}

\urladdr{http://www.math.uni-muenster.de/u/aaron.tikuisis}
\email{a.tikuisis@uni-muenster.de}

\begin{document}
\begin{abstract} 
We construct a simple 
$\mathrm{C}^*$-algebra with nuclear dimension zero that is not isomorphic to 
its tensor product with the Jiang--Su algebra $\cZ$, and  a hyperfinite II$_1$ factor 
not isomorphic to its  tensor product with the separable hyperfinite  II$_1$ factor $\calR$. 
The proofs use a weakening of the Continuum Hypothesis. 
\end{abstract} 

\maketitle

Elliott's program of classification of nuclear (a.k.a.\ amenable) $\mathrm{C}^*$-algebras recently 
underwent  a transformative phase (see e.g., \cite{EllTo:Regularity}). 
Following the counterexamples of R\o rdam and Toms to the original program,
it was realized that a regularity assumption stronger than the nuclearity is necessary for 
$\mathrm{C}^*$-algebras to be classifiable by $K$-theoretic invariants.
Conjecturally, regularity properties of three different flavours are all equivalent and are, modulo the UCT, sufficient for classification (restricting, say, to simple, separable, nuclear $\mathrm{C}^*$-algebras).
We shall consider two of these regularity assumptions on a $\mathrm{C}^*$-algebra $A$. 
One of them 
asserts that $A$ is \emph{$\cZ$-stable}, meaning that it is isomorphic to its tensor product with the Jiang-Su algebra $\cZ$.
An another postulates  that $A$ has finite \emph{nuclear dimension}, 
this being a strengthening of the Completely Positive Approximation Property (CPAP) 
introduced by Winter and Zacharias  in~\cite{WinterZacharias:NucDim}
(the CPAP is an equivalent formulation of amenability for $\mathrm{C}^*$-algebras, see \cite[Chapter 2]{BrOz:C*}). 
The Toms--Winter conjecture states  
(among other things) that 
for separable, nuclear, simple, non-type I $\mathrm{C}^*$-algebras, having finite nuclear dimension is equivalent 
to being $\cZ$-stable (see e.g., \cite{Win:Ten,ToWin:Elliott}). 
The direct implication is a theorem of Winter~(\cite{Win:Nuclear}).  
We show that it badly fails if the separability assumption is dropped. 

\begin{thm}\label{T0}
The Continuum Hypothesis implies that there exists   
a simple nuclear $\mathrm{C}^*$-algebra with nuclear dimension zero which is not $\cZ$-stable. 
\end{thm}

The paradigm of regularity properties for $\mathrm{C}^*$-algebras parallels certain older ideas in the study of von Neumann algebras.
It has long been known that amenability for von Neumann algebras is equivalent to hyperfiniteness, and in the separable, non-type I case, it implies $\calR$-stability (the property of being isomorphic to one's tensor product with the unique separable hyperfinite II$_1$ factor $\calR$; von Neumann algebras with this property are commonly called McDuff) (see \cite[Chapters XIV and XVI]{Tak:TheoryIII}).
It is also not a stretch to find that amenability is equivalent to a von Neumann-theoretic analogue of nuclear dimension zero \cite[Lemma 1.2]{HirKirWhi}.
The construction used to prove Theorem \ref{T0} adapts easily to the von Neumann case, allowing us to prove the following.

\begin{thm}\label{T0-vN}
The Continuum Hypothesis implies that there exists   
a hyperfinite II$_1$ factor which is not $\calR$-stable.
\end{thm}

We prove these theorems by constructing algebras whose central sequence algebras are abelian.
In fact, we strengthen the construction in two directions, one in which the conclusion is strengthened (Theorems~\ref{T2} and \ref{T2-vN}) and another in which weaker set-theoretic axioms are assumed (Theorems~\ref{T1} and \ref{T1-vN}).
It seems unlikely that the conclusions of Theorems~\ref{T1} and \ref{T1-vN} --- or even the stronger conclusions of Theorems~\ref{T2} and \ref{T2-vN} --- are independent from ZFC.
The assumption of Theorems~\ref{T1} and \ref{T1-vN} is the cardinal equality $\fb=\fc$, where $\fb$ denotes the bounding number --- see \S\ref{S.2}, where  
we also give two new characterizations of $\fb$. 

A result similar to Theorem~\ref{T2-vN} was proven in 
\cite[Proposition~3.7 (1)]{FaHaSh:Model3}, where 
a II$_1$ factor $M$ is constructed using ZFC alone, 
whose central sequence is trivial, but every separable II$_1$ subfactor
of $M$ has property $\Gamma$
(this factor has density character $\aleph_1$, with respect to the strong operator topology).
Unlike the example in Theorem~\ref{T2-vN} (which shares the aforementioned properties), the von Neumann algebra in \cite[Proposition~3.7 (1)]{FaHaSh:Model3} is not hyperfinite, due to the use of the free product construction.
Likewise, a minor modification of the construction in \cite{FaHaSh:Model3} gives a monotracial $\mathrm{C}^*$-algebra whose central sequence algebra is abelian, yet whose non-type I, separable subalgebras each have nonabelian central sequence algebras; but again, the free product construction prevents this example from being nuclear, let alone having nuclear dimension zero.

This factor in Theorem~\ref{T0-vN} is clearly different from the `obvious' hyperfinite factor whose predual    has character
density $\aleph_1$. The existence of `nonobvious' hyperfinite II$_1$ factors with nonseparable preduals 
was first proved by Widom (\cite{Wid:Nonisomorphic}), 
and in \cite{FaKa:NonseparableII} it was proved that there are $2^\kappa$ nonisomorphic 
hyperfinite II$_1$ factors with predual of density character $\kappa$, for every 
uncountable cardinal~$\kappa$.

In \S\ref{S.1} we provide more information about the algebras we  construct and their central sequence algebras. 
In \S\ref{S.2} we give two new reformulations of the bounding number $\fb$, one in terms of 
 convergent   subseries of null sequences and another in terms of convergence of a sequence of 
 inner automorphisms of a $\mathrm{C}^*$-algebra. 
 The proofs of main results are completed in \S\ref{S.3}, using 
  transfinite recursive construction,

  and   in \S\ref{S.4} we give concluding remarks. 

Recall that \emph{density character} of a topological space $X$ is the minimal 
cardinality of a dense subset. Therefore $X$ is separable iff its density character is $\aleph_0$
and it has cardinality $\fc$ if and only if its density character is $\leq \fc$.
In our case, $X$ will either be a $\mathrm{C}^*$-algebra equipped with the norm topology, or a von Neumann algebra considered with the strong operator topology (SOT).
In fact, the density character of a von Neumann algebra $M$ is the same whether measured with respect to the strong operator topology or any other usual, non-norm, von Neumann algebra topology (eg.\ WOT, weak$^*$-topology); it also coincides with the density character of the predual $M_*$ under its norm topology.
All $\mathrm{C}^*$-algebras are assumed to be unital. 
Background 
can be found e.g., in \cite{Black:Operator} (for $\mathrm{C}^*$-algebras and von Neumann algebras) and 
in \cite{Bla:Cardinal} (for set theory).

\subsection*{Acknowledgments} 
The results of this paper were proven during the September 2012 Fields Institute Workshop on Applications to Operator Algebras, 
November 2012 Oberwolfach Workshop on $\mathrm{C}^*$-algebras, Dynamics and Classification, 
and in November 2012 back at the Fields Institute. 
I.F.\ is partially supported by NSERC,  D.H.\ is supported by the Fields Institute, and 
A.T.\ is supported by DFG (SFB 878).

\section{Locally matricial algebras}
\label{S.1}

Following \cite{FaKa:Nonseparable} and \cite{Kat:Non-separable}, we say that 
a $\mathrm{C}^*$-algebra $A$ is 
\begin{itemize}
\item \emph{approximately matricial}
(or \emph{AM})
if it has a directed family of
full matrix subalgebras with dense union.
\item \emph{locally matricial}
(or \emph{LM})
if for any finite subset $\F$ of $A$ and any $\e > 0$,
there exists a  full matrix subalgebra $M$ of $A$
such that for every $a\in \F$ we have $\dist(a,M)<\e$.
\item \emph{locally finite dimensional} (or \emph{LF})
if for any finite subset $\F$ of $A$ and any $\e > 0$,
there exists a  finite dimensional subalgebra $M$ of $A$
such that for every $a\in \F$ we have $\dist(a,M)<\e$.
\end{itemize}
Nuclear dimension, as defined in \cite{WinterZacharias:NucDim}, is a property that is preserved under local approximation; it follows that each LF algebra has nuclear dimension zero (in fact, the converse also holds \cite{WinterZacharias:NucDim}).
LM algebras are not necessarily AM (\cite[Theorem~1.5]{FaKa:Nonseparable}), although LM algebras are direct limits of separable UHF algebras (\cite[Lemma~2.12]{FaKa:Nonseparable}) 
LM algebras are simple.

Recall that $x_n$, for $n\in \bbN$, is a  \emph{central sequence} of a $\mathrm{C}^*$-algebra $A$
if $\|[x_n,a]\|\to 0$ for all $a\in A$. The \emph{central sequence algebra} of $A$ is a subalgebra of 
$\ell^\infty(A)/c_0(A)$ consisting of all equivalence classes of (bounded) central sequences.

\begin{lemma} \label{L0} 
If a central sequence algebra of a unital algebra 
$A$ is abelian then $A$ is not $\cZ$-stable.  
\end{lemma} 

\begin{proof} 
If $B$ and $C$ are unital $\mathrm{C}^*$-algebras then 
the central sequence algebra of $B\otimes_\nu C$
 (where $\otimes_\nu$ is any $\mathrm{C}^*$-tensor 
product) clearly has the central sequence algebra of $C$ as a unital subalgebra. The conclusion now follows from the fact that the central sequence algebra of $\cZ$ has a subalgebra isomorphic to~$\cZ$. 
\end{proof}

Every sequence of elements from the centre of a $\mathrm{C}^*$-algebra is obviously a central sequence.
We call the central sequence of a $\mathrm{C}^*$-algebra $A$ \emph{trivial} if it only consists of (equivalence classes of) such sequences (i.e.\ if it equals $(\ell^\infty(Z(A))+c_0(A))/c_0(A)$).

\begin{theorem}\label{T2} The Continuum Hypothesis implies that there exists
an AM algebra  all of whose central sequences are  trivial. 
\end{theorem}

By Lemma~\ref{L0}, Theorem~\ref{T0} will follow once Theorem~\ref{T2} is proved in \S\ref{S.3}.  

Strengthening Theorem~\ref{T0} in another direction we show that an assumption weaker than both the Continuum Hypothesis and  Martin's Axiom suffices. 
We write  $\fc=2^{\aleph_0}$ and $\fb$ for the bounding number (see  \S\ref{S.2} for definitions).

\begin{theorem} \label{T1} 
If $\fb=\fc$ then   there exists a LM algebra whose central sequence algebra is abelian. 
\end{theorem}

By Lemma~\ref{L0} and Theorem~\ref{T2} (once the latter is proved) 
 the conclusion of Theorem~\ref{T0} also follows from $\fb=\fc$. 

A sequence, $x_n$ for $n \in \bbN$, in a von Neumann algebra $M$ is said to be central if $[x_n,a] \to 0$ in the strong operator topology, for every $a \in M$.
The central sequence algebra of $M$ is the subalgebra of
\[ \ell^\infty(A)/\{(x_n) | \mathrm{SOT}-\lim x_n = 0\} \]
consisting of equivalence classes of central sequences.
A trivial central sequence is one in the same equivalence class as a sequence from $Z(M)$.
Proven in the same manner, we have the following analogue of Lemma \ref{L0}.

\begin{lemma} \label{L0-vN}
If the central sequence algebra of a von Neumann algebra $M$ is abelian then $M$ is not $\calR$-stable.
\end{lemma}

We will make great use of the fact that, when $M$ has a faithful trace $\tau$, then the strong operator topology coincides with the topology induced by the norm $\|\cdot\|_2$, given by
\[ \|x\|_2 := \tau(x^*x)^{1/2}; \]
indeed, adapting our main construction from the $\mathrm{C}^*$-setting to the von Neumann setting will primarily consist in using $\|\cdot\|_2$ in place of $\|\cdot\|$.
Here are the von Neumann versions of Theorems \ref{T2} and \ref{T1}.

\begin{theorem}\label{T2-vN} The Continuum Hypothesis implies that there exists
a hyperfinite II$_1$ factor all of whose central sequences are  trivial. 
\end{theorem}

\begin{theorem} \label{T1-vN}
If $\fb=\fc$ then   there exists a hyperfinite II$_1$ factor whose central sequence algebra is abelian, and in particular, which is not $\calR$-stable.
\end{theorem}

\section{New characterizations of the bounding number} 
\label{S.2}

In the present section we provide alternative characterizations of the bounding number, $\fb$.
Readers interested only in the proofs of Theorems~\ref{T2} and \ref{T2-vN}
can skip ahead to \S\ref{S.3}, modulo the  facts that $\fb_1$ defined below is uncountable (this is not hard to show directly, or alternatively, using $\fb_2 \leq \fb_1$ and the uncountability of $\fb_2$). 
Given two functions $f, g : \bbN \to \bbN$,
 we say that $f$ \emph{eventually dominates} $g$
 if $f(n) \ge g(n)$ for all but finitely many
 $n \in \bbN$.

Recall
that the \emph{bounding number}, which we denote by $\fb$, is the minimal cardinality $\kappa$ such that there exist functions $f^\xi:\bbN \to \bbN$ for $\xi < \kappa$, such that for every $g:\bbN \to \bbN$, there exists some $f^\xi$ which isn't eventually dominated by $g$.

Let $\fb_1$ be the minimal cardinal $\kappa$ such that there exists a unital $\mathrm{C}^*$-algebra $A$ of density character $\kappa$ and a central sequence of unitaries $u_n$, for $n\in \bbN$, such that 
the following holds. For every  subsequence $u_{n(i)}$, for $i\in \bbN$, with 
\[
v_k:=u_{n(1)} u_{n(2)} \dots u_{n(k)},
\]
we have that the sequence of inner automorphisms $\Ad v_k$, for $k\in \bbN$, does not converge pointwise 
on $A$. 

Let $\fb_1(\LF)$ be the minimal $\kappa$ with the above property where we additionally 
require that 
$A$ is LF. 

Let $\fb_2$ be the minimal cardinal $\kappa$ such that there are sequences $r^\xi_n$, for $n\in \bbN$
and $\xi<\kappa$, satisfying 
\begin{enumerate}
\item\label{r1}  $r^\xi_n\in (0,\infty)$ for all $\xi$ and $n$, 
\item\label{r2}  $\lim_n r^\xi_n=0$ for all $\xi$, and
\item there is no infinite increasing sequence $n(i)$, for $i\in \bbN$, of natural numbers such that 
$\sum_i r^\xi_{n(i)}<\infty$ for all $\xi$. 
\end{enumerate}
It is well-known that Martin's Axiom for $\sigma$-centered posets implies $\fb=\fc$ (see e.g., \cite{Fr:MA}).   

Readers familiar with the generalized Galois--Tukey connections will 
have no trouble in recasting the above definitions in this framework and observing 
that each part of Lemma~\ref{L.ConvBound} asserts the existence of a morphism in the terminology 
of \cite[\S 4]{Bla:Cardinal}.

\begin{prop} \label{P.b} We have $\fb = \fb_1=\fb_1(\LF)=\fb_2$. 
\end{prop}

The proof that $\fb = \fb_2$ uses the following lemma, which relates null sequences to functions $\bbN \to \bbN$. 

\begin{lemma}\label{L.ConvBound}
\mbox{} 
\begin{enumerate}
\item\label{L.ConvBound1}
 If $r_n \in (0,\infty)$, for $n \in \bbN$, is a null sequence, then there exists a function $f:\bbN \to \bbN$ such that for any $g:\bbN \to \bbN$, if $g$ eventually dominates $f$ then
$ \sum_{i=1}^\infty r_{g(i)} < \infty$. 
\item\label{L.ConvBound2}
 If $f:\bbN \to \bbN$ is a function than there exists a null sequence $r_n \in (0,\infty)$, for $n \in \bbN$, such that for any $g:\bbN \to \bbN$, if
$ \sum_{i=1}^\infty r_{g(i)} < \infty$
then $g$ eventually dominates $f$.
\end{enumerate}
\end{lemma}

\begin{proof}
\ref{L.ConvBound1}
Let $f$ be any function such that
 $$(\forall i)(\forall n \ge f(i))
 [ r_n \le \frac{1}{2^i}],$$
(which exists because
 $\lim_n r_n = 0$).
If $g:\bbN \to \bbN$ eventually dominates $f$, then eventually $r_{g(i)} < 2^{-i}$ and so
$ \sum_{i=1}^\infty r_{g(i)} < \infty$. 

\ref{L.ConvBound2}
Let $r_n$ be any sequence of positive real numbers
 converging to $0$ satisfying
 $$(\forall i)(\forall n \le f(i))
 [ r_n \ge \frac{1}{i} ].$$
Aiming to prove the contrapositive, suppose that $g:\bbN \to \bbN$ does not eventually dominate $f$.
Let $B \subseteq \mathbb{N}$ be the
 set of all $i$ for which
 $f(i) > g(i)$.
Next, ``thin out'' $B$ as follows:
 let $b_j$ be an increasing sequence of
 distinct elements of $B$
 such that for each $j$,
 \textit{over half} of the elements of
 $\textnormal{Im}(g)$ before $g(b_{j+1})$ are not
 before $g(b_j)$.
Then,
\begin{align*}
 \sum_{i=0}^\infty r_{g(i)} & =
 \sum_{i < b_0} r_{g(i)} +
 \sum_{b_0 \le i < b_1} r_{g(i)} +
 \sum_{b_1 \le i < b_2} r_{g(i)} + \dots \\
& \ge
 \sum_{i < b_0} \frac{1}{b_0} +
 \sum_{b_0 \le i < b_1} \frac{1}{b_1} +
 \sum_{b_1 \le i < b_2} \frac{1}{b_2} + \dots \\
& \ge
 (b_0)\frac{1}{b_0} +
 (\frac{1}{2}b_1)\frac{1}{b_1} +
 (\frac{1}{2}b_2)\frac{1}{b_2} + \dots \\
& =
 1 + \frac{1}{2} + \frac{1}{2} + \dots \\
& =
 \infty,
\end{align*}
as required.
\end{proof}
\begin{proof}[Proof that $\fb = \fb_2$.]
$\fb \leq \fb_2$:
Let $\kappa < \mathfrak{b}$ be arbitrary.
Let $r^\xi_n$ be sequences of positive reals,
 for $n \in \mathbb{N}$ and $\xi < \kappa$.
Assume that $\lim_n r^\xi_n = 0$ for each $\xi < \kappa$.
For each $\xi < \kappa$, let $f_\xi:\bbN \to \bbN$ be as in Lemma \ref{L.ConvBound} \ref{L.ConvBound1}.
Since $\kappa < \mathfrak{b}$,
 there is some $g:\bbN \to \bbN$
 that eventually dominates each $f_\xi$, and therefore by Lemma \ref{L.ConvBound} \ref{L.ConvBound1}, $\sum_i r^\xi_{g(i)} < \infty$ for each $\xi < \kappa$.

The proof that $\fb_2 \leq \fb$ is exactly the same, but this time using Lemma \ref{L.ConvBound} \ref{L.ConvBound2} instead of \ref{L.ConvBound1}.
\end{proof}

To prove the remaining part of Proposition \ref{P.b} (namely, that $\fb_1 = \fb_1(\LF) = \fb_2$), we need the following lemmas.

\begin{lemma}\label{L.RT}
Let $(r_n)_{n=1}^\infty$ be a null sequence of real numbers.
Then $\sum_{n=1}^\infty r_n$ converges in $\bbR$ if and only if $(\sum_{n=1}^N r_n \mod 1)_{N=1}^\infty$ converges in $\bbR/\bbZ$.
\end{lemma}

\begin{proof}
The forward implication is trivial.
Suppose, on the other hand, that $(\sum_{n=1}^N r_n \mod 1)_{N=1}^\infty$ converges in $\bbR/\bbZ$.
Let $N_0$ be such that $|r_n| < 1/2$ for $n \geq N_0$.
If $N_0$ is sufficiently large, then for all $N \geq N_0$,
\[\textstyle \sum_{n=N_0}^N r_n \in (-\e,\e) + \bbZ. \]
If $\e < 1/2$ then, between these facts, it follows that
\[ \textstyle\sum_{n=N_0}^N r_n \in (-\e,\e). \]
This shows that the partial sums form a Cauchy sequence, so that 
$ \sum_{n=1}^\infty r_n$
converges in $\bbR$.
\end{proof}

The following simple fact will be used both here and in Section \ref{S.3}, to `decentralize' elements or sequences.

\begin{lemma}\label{L.endo} 
If $A$ is a $\mathrm{C}^*$-algebra and $\alpha:A \to A$ is an endomorphism, then
$\alpha^+\colon A\to M_2(A)$ defined by 
\[
\alpha^+(a)=\begin{pmatrix} a & 0 \\ 0 & \alpha(a) \end{pmatrix}
\]
is a $^*$-homomorphism. 
Moreover, for any $a \in A$,
\[ \|[\alpha^+(a),v]\| = \|a-\alpha^+(a)\|, \]
where
\begin{equation}
\label{L.endo.E.v}
 v := \begin{pmatrix} 0_A & 1_A \\ 1_A & 0_A \end{pmatrix}.
\end{equation}
\end{lemma}  

\begin{proof} Obvious. 
\end{proof} 

\begin{lemma}\label{L.CarSeq}
Let $(r_n)$ be a null sequence of real numbers.
Then there exist unitaries $v,u_n$, for $n\in \bbN$,  in the CAR algebra such that $(u_n)$ is a central sequence and, for any increasing sequence $(n(i))$ of natural numbers,
\[\textstyle \sum_i r_{n(i)} \]
converges in $\bbR$ if and only if
\[ \Ad(u_{n(1)} \dots u_{n(k)})(v) \]
converges.
Moreover, the CAR algebra contains subalgebras $A_n \cong M_2$ such that $u_n \in A_n$ and $[u_n,A_m] = 0$ if $m \neq n$.
\end{lemma}

\begin{proof}
Let $A=M_2^{\otimes \infty}$ and let $\alpha$ be the automorphism of $A$ given by applying
\[ \Ad \left( \begin{array}{cc} 0 & 1 \\ 1 & 0 \end{array} \right) \]
to each tensor factor $M_2$.
Then $B = M_2(A)$ is isomorphic to the CAR algebra, and we shall make use of the morphism $\alpha^+:A \to B$ given by Lemma \ref{L.endo}.

Define $v$ as in \eqref{L.endo.E.v}.
For each $n$, set
\[ u_n' := 1_{M_2}^{\otimes (n-1)} \otimes
\left( \begin{array}{cc} 1 & 0 \\ 0 & e^{2\pi i r_n} \end{array} \right)
\otimes 1_{M^2}^{\otimes \infty} \in A, \]
and
\[ u_n = \alpha^+(u_n'). \]
Evidently, $(u_n')$ is central in $A$, and an easy computation shows that $(u_n)$ approximately commutes with $v$, and therefore, with every element of $B$.
That is to say, $(u_n)$ is a central sequence.

Now, given an increasing sequence $(n(i))$ of natural numbers, we compute
\[ \textstyle\Ad(u_{n(1)}\dots u_{n(i)}(v) = \exp(2\pi i \sum_{j=1}^i r_{n(j)}) v, \]
and therefore, the sequence $(\Ad(u_{n(1)}\dots u_{n(i)}(v))$ converges if and only if the sequence
\[ \textstyle(\exp(2\pi i \sum_{j=1}^i r_{n(j)})) \]
converges in $\bbT$, which by Lemma \ref{L.RT}, occurs if and only if
$ \sum_{i=1}^\infty r_{n(i)} $
converges in $\bbR$.

Finally, we set
$ A_n := \alpha^+(1_{M_2}^{\otimes (n-1)} \otimes
M_2 \otimes 1_{M^2}^{\otimes \infty})$. 
\end{proof}

\begin{proof}[Proof of Proposition \ref{P.b}] We shall now prove that 
 $\fb_1 = \fb_1(\LF) = \fb_2$.

$\fb_2 \leq \fb_1$:
Let $A$ be a unital 
$\mathrm{C}^*$-algebra of density character $\kappa<\fb_2$ and let $v_n$, for $n\in \bbN$, be a
central sequence of unitaries in~$A$. 
Let $a_\xi$, for $\xi<\kappa$, be a dense subset of $A$ and let
\[
r^\xi_n=\|[v_n,a_\xi]\|. 
\]
Since $(v_n)$ is a central sequence, we have  $\lim_n r^\xi_n=0$ for all $\xi$. 
We can therefore choose an increasing sequence of natural numbers $n(i)$, for $i\in \bbN$, 
so that $\sum_i r^\xi_{n(i)}<\infty$ for all $\xi$. 

Let  $w_k=v_{n(1)} v_{n(2)} \dots v_{n(k)}$. For $k<m$ and all $\xi$ 
we have 
\begin{align*}
\|(\Ad w_k)a_\xi-(\Ad w_m)a_\xi\|
&= \|\Ad w_k(a_\xi - \Ad(v_{n(k+1)}\dots v_{n(m)})(a_\xi))\| \\
&= \|[a_\xi,v_{n(k+1)}\dots v_{n(m)}]\| \\
\textstyle &\leq \sum_{j=k}^{m-1} 
\|(\Ad w_j)a_\xi- (\Ad w_{j+1}) a_\xi\|\\
\textstyle &\leq \sum_{j=k}^{m-1} (r^\xi_{n(j)} + 2^{-j}\|a_\xi\|), 
\end{align*}
and therefore $(\Ad w_k) a_\xi$, for $k\in \bbN$, is a Cauchy sequence. 
Since the automorphisms $(\Ad w_k)$ are isometries which 
pointwise converge on a dense subset of $A$, they 
pointwise converge to an endomorphism. 
Since $A$ was arbitrary we have proved that $\kappa<\fb_2$ implies $\kappa<\fb_1$, 
and therefore $\fb_1\geq \fb_2$. 

$\fb_1(\LF)\geq \fb_1$: this is trivial.

$\fb_2\geq \fb_1(\LF)$:
Let $\kappa < \fb_2$, and let $r_n^\xi \in (0,\infty)$, for $n \in \mathbb{N}$, be a null sequence, for each $\xi < \kappa$.
For each $\xi < \kappa$, let $A^{(\xi)}$ be a copy of the CAR algebra and let
\[ v^{(\xi)},u_n^{(\xi)} \in A^{(\xi)} \]
be unitaries, for $n \in \bbN$, as given by Lemma \ref{L.CarSeq}.
Also, let $A_n^{(\xi)} \subset A^{(\xi)}$, for $n \in \bbN$, be the subalgebras given by the same lemma.

Set $B = \prod_{\xi < \kappa} A^{(\xi)}$, and define
\[ u_n := (u_n^{(\xi)})_{\xi < \kappa} \in A \]
for each natural number $n$ and
\[ \hat v^{(j)} := (\delta_{ij} v^{(\xi)})_{i \in I} \in A. \]
Set
\[\textstyle B_n := \prod_{i \in I} A_n^{(\xi)} \]
for each natural number $n$, and set
\[\textstyle A' := \mathrm{C}^*\left(\bigcup_{n \in \bbN} B_n \cup \bigoplus_{i \in I} A^{(\xi)}\right); \]
We easily see that $A'$ is LF and contains each $u_n$ and each $v^{(\xi)}$, and that $(u_n)$ is a central sequence in $A'$.
By a downward L\"owenheim-Skolem argument, there exists an LF subalgebra $A$ of $A'$ with density character $\kappa$ and which contains each $v^{(\xi)}$ and each $u_n$.

Now, suppose that $(n(i))$ is an increasing sequence of natural numbers such that
\[ \Ad(u_{n(1)} \dots u_{n(k)}) \]
converges in the point-norm topology.
Then in particular,
\[ \Ad(u_{n(i)} \dots u_{n(k)})(v^{(\xi)}) \]
converges for each $\xi < \kappa$, which by Lemma \ref{L.CarSeq}, means that
$ \sum_{i=1}^\infty r^{\xi}_{n(i)} < \infty$. 

The proof is complete.
\end{proof}

\begin{remark}
\label{R.b-vN}
The proof that $\fb_1 \leq \fb_2$ ($=\fb$) can be adapted (by using $\|\cdot\|_2$ in place of $\|\cdot\|$) to show that, if $M$ is a von Neumann algebra with a faithful trace and with density character $< \fb$, and $u_n$, for $n \in \bbN$, is a central sequence from $M$, then there exists a subsequence $u_{n(i)}$, for $i \in \bbN$, such that
\[ \Ad u_{n(1)} \dots u_{n(k)} \]
converges (in the point-strong operator topology).
However, the proof of the converse does not adapt, since when the construction used to show $\fb_2 \leq \fb_1(LF)$ is adapted to the von Neumann setting, the resulting von Neumann algebra does not have a faithful trace.
\end{remark}

\section{Proofs of the main theorems}
\label{S.3} 

\begin{lemma} \label{L.ZZ}
Let $A$ be an LM algebra, and for each $i=1,\dots,N$, let $x_n^{(i)}$, for $n \in \mathbb{N}$, be a central sequence.
Then there exists an increasing sequence $n(k) \in \mathbb{N}$, for $k \in \mathbb{N}$, and $y_k^{(i)}$ such that:
\begin{enumerate}
\item for each $i$, $\|y_k^{(i)} - x_{n(k)}^{(i)}\| \to 0$ as $k \to \infty$; and
\item for each $i,i',k,k' \in \mathbb{N}$, if $k \neq k'$ then $y_k^{(i)}$ and $y_{k'}^{(i')}$ $^*$-commute.
\end{enumerate}
\end{lemma}

\begin{proof}
Since $A$ is LM, there exists a separable LM subalgebra $B$ which contains each $x_n^{(i)}$.
By perturbing the sequences $x_n^{(i)}$ (by an error that vanishes at $\infty$), we have without loss of generality that
\[ B = M_{m_1} \otimes M_{m_2} \otimes \dots, \]
and $x_n^{(i)} \in M_{m_1} \otimes \dots \otimes M_{m_n}$ for each $i$ and $n$.

Let $\e_n > 0$ be any null sequence.
Set $n(1)=1$.
Using compactness of the unit ball of a matrix algebra, we may iteratively choose $n(k)$ such that
\begin{equation}
\|[x_{n(k)}^{(i)},a]\| \leq \e_n\|a\|
\label{L.ZZ.eq}
\end{equation}
for all $a \in M_{n(1)} \otimes \dots \otimes M_{n(k-1)}$.
Letting $E_k$ denote the conditional expectation from $M_{n(k)}$ to $M_{n(k-1)}' \cap M_{n(k)}$, set
\[ y_k^{(i)} := E_k(x_{n(k)}^{(i)}). \]
Then, it follows from \eqref{L.ZZ.eq} (and by writing the conditional expectation as an average over the unitary group of $M_{n(k-1)}$) that
\[ \|y_k^{(i)}-x_{n(k)}^{(i)}\| \leq \e_n, \]
as required.
\end{proof}

We say that a central sequence is \emph{hypercentral} if
it commutes with every other central sequence (i.e., if it is a representing sequence
of a central element of the central sequence algebra). 
(Although the terminology originated in theory of II$_1$ factors and our $\textrm{C}^*$-algebras have 
a unique trace, we emphasize that only the operator norm is being used here.)

\begin{lemma} \label{L.central} Assume $A$ is an LM algebra of density character $<\fb$. If 
$x_n$, for $n\in \bbN$, is a  central sequence
which is not hypercentral 
then there exists an endomorphism $\alpha$ of $A$ such that
$\liminf_n \|x_n-\alpha(x_n)\|>0$.
\end{lemma} 

\begin{proof}
Let $u_n$, for $n\in \bbN$, be a central sequence 
such that for some $\e>0$ we have $\|[x_n,u_n]\|> \e>0$ for all $n$.
We may assume each $u_n$ is a unitary since every element in a $\mathrm{C}^*$-algebra is a linear combination of four unitaries.

By Lemma \ref{L.ZZ}, by passing to a subsequence, there exist sequences $y_n$ and $v_n$, for $n \in \bbN$, such that
\[ \|y_n-x_n\|,\|v_n-u_n\| \to 0 \]
as $n \to \infty$, and $[v_n,v_m]$ $^*$-commutes with $v_m,y_m$ for $n \neq m$.

Since $u_n$ is unitary, by functional calculus, we may arrange that $v_n$ is too.
(Note that modifying $v_n$ using functional calculus does not change the fact that it $^*$-commutes with $v_m,y_m$ for $n \neq m$.)

By using that the density character of $A$ is less than $\fb=\fb_1$ (by Proposition~\ref{P.b}),
we can go to a subsequence of  $v_m$ (again denoted $v_m)$ such that 
the automorphisms 
 $\alpha_n:=\prod_{j=1}^n \Ad v_j$  converge pointwise
 to an endomorphism, and so we may set
\[ \alpha :=\lim_k \alpha_k. \]
Then $\alpha(y_n) = \Ad v_n \circ y_n$, and therefore,
\begin{align*}
\liminf_n \|x_n - \alpha(x_n)\|
&= \liminf_n \|y_n - \alpha(y_n)\| \\
&= \liminf_n \|[y_n,v_n]\| \\
&= \liminf_n \|[x_n,u_n]\| > 0,
\end{align*}
as required.
\end{proof} 

Note that by assuming that the algebra $A$ is separable, the assumption in the previous lemma, that $x_n$ is not hypercentral, comes for free.

\begin{prop} \label{P.separable}
Let $A$ be a separable LM algebra.
Then every hypercentral sequence is trivial.
\end{prop}

\begin{proof}
Let $x_n$, for $n \in \bbN$, be a nontrivial central sequence.
By passing to a subsequence, we may assume for some $\e > 0$, we have $d(x_n,Z(A)) > \e$ for all $n$.
Using the proof of Lemma \ref{L.ZZ} (and by passing again to a subsequence and perturbing), we may assume that $A = M_{m_1} \otimes M_{m_2} \otimes \dots$ such that
\[ x_n \in 1_{m_1} \otimes \dots \otimes 1_{m_{n-1}} \otimes M_{m_n}. \]
Since $x_n$ has distance at least $\e$ from the centre of $A$, there must exist
\[ y_n \in 1_{m_1} \otimes \dots \otimes 1_{m_{n-1}} \otimes M_{m_n} \]
such that $\|[x_n,y_n]\| \geq \e$.
Evidently, $y_n$ for $n \in \bbN$, forms a central sequence, and it does not asymptotically commute with the given sequence, as required.
\end{proof}

\begin{proof}[Proof of Theorem~\ref{T1}] 
We need to construct $A$ so that  all central sequences of $A$ are 
hypercentral. 
Fix a surjection $\chi\colon \fc\to \fc^2$ such that if $\chi(\xi)=(\eta,\zeta)$ 
then $\eta\leq \xi$ 
and moreover 
for every fixed pair $(\eta,\zeta)\in \fc^2$ the set $\{\xi: \chi(\xi)=(\eta,\zeta)\}$ is cofinal 
in $\fc$. 
This is possible because  every infinite cardinal $\kappa$ is equinumerous with $\kappa^2$. 

We construct $A$ as a transfinite direct limit of LM algebras $A_\xi$, for $\xi<\gamma$
for some ordinal $\gamma\leq\fc$. Each $A_\xi$ will be of density character $\leq\fc$, 
and therefore the set $(A_\xi)_1{}^{\bbN}$ of all sequences in the unit ball $(A_\xi)_1$ of $A_\xi$
will have cardinality $\fc$. For each~$\xi$ we fix an enumeration 
$(\vec x(\xi,\eta): \eta<\fc)$ of $(A_\xi)_1{}^{\bbN}$ as soon as this algebra is defined. 

We now describe the recursive construction of a directed system of LM-algebras
$A_\xi$, $\beta_{\xi\eta}\colon A_\xi\to A_\eta$ for $\xi<\eta$. 

Let $A_0=M_{2^\infty}$. If $\delta$ is a limit ordinal and $A_\xi$, for $\xi<\delta$ are defined, 
we let $A_\delta=\lim_\xi A_\xi$, the inductive limit of $A_\xi$. 

Now  assume $A_\xi$ is defined and we construct $A_{\xi+1}$. 
If all central sequences in $A_\xi$ are hypercentral, we stop our recursive construction and let 
$A=A_\xi$. 

Otherwise, write  $\chi(\xi)=(\eta,\zeta)$. Since $\eta\leq \xi$ the algebra 
$A_\eta$ was already defined  we can consider the sequence $\vec x(\eta,\zeta)$ in $A_\eta$. 
Let $x_j$, for $j\in \bbN$, be the $\beta_{\eta,\xi}$-image of this sequence in $A_\xi$.

If this is not a central sequence, or if it is a hypercentral sequence, 
let $A_{\xi+1}=A_\xi$
(actually we can do almost anything here). 

Now assume  this sequence is central, but not hypercentral. 
Use  Lemma~\ref{L.central} to find $\alpha\in \Aut(A_\xi)$ and $\e>0$ such that 
$\|x_j-\alpha(x_j)\|>\e$  for infinitely many~$j$. Then let $A_{\xi+1}=M_2(A_\xi)$ and 
apply Lemma~\ref{L.endo} to find 
a *-homomorphism $\beta_\xi\colon A_\xi\to A_{\xi+1}$  such that $\alpha(x_j)$, for $j\in \bbN$, 
is not a central sequence in $A_{\xi+1}$.

This describes the recursive construction of transfinite directed system of LM algebras 
of length $\fc$. If the construction does not stop at any stage $\xi$, 
let $A=\lim_\xi A_\xi$. 

We claim that all central sequences of $A$ are hypercentral. 
For simplicity of notation 
assume that each $A_\xi$ is a unital subalgebra of $A$ (this is not a problem since
all connecting maps had trivial kernels) so that each $\beta_{\xi\eta}$ is equal to 
the identity on $A_\xi$. 
Assume otherwise and let $x_j$, for $j\in \bbN$, be a central, non-hypercentral, sequence. 
Also fix a central sequence $y_j$, for $j\in \bbN$, such that $[x_j,y_j]\not \to 0$. 

Since the cofinality of $\fc$ is uncountable, these sequences are
included in $A_\eta$ for some $\eta<\fc$. 
The first one  was enumerated as $\vec x(\eta,\zeta)$ for some $\zeta<\fc$. 
Since $\chi$ is a surjection, there is $\xi$ such that $\chi(\xi)=(\eta,\zeta)$. 
We also have $\xi\geq \eta$ by the choice of $\chi$. 
By the construction, $A_{\xi+1}$ was defined so that (the image of) 
$x_j$, for $j\in \bbN$, is not a central sequence. This contradiction completes the proof. 
\end{proof}

\begin{proof}[Proof of Theorem~\ref{T2}] The proof is essentially identical to the proof of 
Theorem~\ref{T1}. The only difference is that, since $\fc=\aleph_1$, all algebras $A_\xi$ for $\xi<\fc$ 
are separable and we can therefore use Proposition~\ref{P.separable} and Lemma \ref{L.central} to assure that the central 
sequence algebra of the limit is trivial. The constructed algebra is an LM algebra
of density character $\aleph_1$ and is therefore AM by \cite[Theorem~1.5]{FaKa:Nonseparable}. 
\end{proof} 

\begin{proof}[Proof of Theorems~\ref{T1-vN} and \ref{T2-vN}]
We can see that the analogue of Lemma~\ref{L.ZZ} for hyperfinite II$_1$ factors holds, by applying the lemma to a dense LM $\mathrm{C}^*$-subalgebra.
Using this and Remark~\ref{R.b-vN} in the proof of Lemma~\ref{L.central} allows us to adapt that lemma to the von Neumann case, showing that if $M$ is a von Neumann algebra of density character $<\fb$ and $(x_n)$ is a (SOT-)central sequence which is not (SOT-)hypercentral, then there exists an endomorphism $\alpha$ of $M$ such that
\[ \liminf_n \|x_n - \alpha(x_n)\|_2 > 0. \]
The analogue of Proposition~\ref{P.separable} for $\calR$ holds, by using an appropriate dense LM $\mathrm{C}^*$-subalgebra.
Finally, using these von Neumann-theoretic adaptations (and the strong operator topology in place of the norm topology), the proofs of Theorems~\ref{T1} and \ref{T2} become proofs of Theorems~\ref{T1-vN} and \ref{T2-vN} respectively.
\end{proof}

\section{Concluding remarks}
\label{S.4}
The following is a well-known result about 
separable $\mathrm{C}^*$-algebras (see e.g.,~\cite{Win:Ten}).

\begin{theorem}
Let $A$ be a unital, separable $\mathrm{C}^*$-algebra.
The following are equivalent.
\begin{enumerate}
\item $A$ is $\cZ$-stable;
\item $\cZ$ embeds into the central sequence algebra of $A$.
\item for any finite subsets $\mathcal{F}, \mathcal{G}$ of $A, \cZ$ respectively, and any $\e > 0$, there exists an $(\mathcal{G},\e)$-approximately multiplicative $*$-linear unital map $\phi:\cZ \to A$ such that
\[ \|[\phi(x),y]\| < \e \]
for all $x \in \mathcal{G}$ and $y \in \mathcal{F}$;
\end{enumerate}
\end{theorem}

In the nonseparable case, (1) $\Rightarrow$ (2) $\Rightarrow$ (3), and Theorem~\ref{T1} shows that (3) $\not\Rightarrow$ (2). 
Note that, by a downward L\"owenheim-Skolem argument, (3) remains equivalent to ``local $\cZ$-stability'' in the following sense:
\begin{enumerate}
\item[(1')] For every separable set $X \subseteq A$, there exists a $\cZ$-stable subalgebra $B \subseteq A$ which contains $X$.
\end{enumerate}
Whether (2) $\Rightarrow$ (1) holds remain unclear.

\begin{question}
Is there a (nonseparable) non-$\cZ$-stable $\mathrm{C}^*$-algebra $A$ such that $\cZ$ embeds into its central sequence algebra?
\end{question}

\begin{question} Can the conclusions of Theorem~\ref{T1}, Theorem~\ref{T2} and Theorem~\ref{T2-vN} be proved in ZFC? 
\end{question}

An approach to these two questions alternative to transfinite recursion 
would be to show that such algebra could be directly 
defined using combinatorics of the uncountable and some of the constructions from \cite{FaKa:Nonseparable} or \cite{Fa:Graphs}.

It is  not clear whether the conclusion of Theorem~\ref {T1} is genuinely weaker than the conclusion of Theorem~\ref{T2}, since we don't know whether there exists a simple C*-algebra whose central sequence algebra is abelian and nontrivial. Martino Lupini pointed out that such algebra cannot be separable unless the simplicity assumption is dropped. 
On the other hand, it is well-known that 
there exists a II$_1$ factor with a separable predual whose central sequence 
algebra is abelian and nontrivial (\cite{DiLa:Deux}).

\providecommand{\bysame}{\leavevmode\hbox to3em{\hrulefill}\thinspace}
\providecommand{\MR}{\relax\ifhmode\unskip\space\fi MR }
\providecommand{\MRhref}[2]{%
  \href{http://www.ams.org/mathscinet-getitem?mr=#1}{#2}
}
\providecommand{\href}[2]{#2}


\begin{thebibliography}{10}

\bibitem{Black:Operator}
B.~Blackadar, \emph{Operator algebras}, Encyclopaedia of Mathematical Sciences,
  vol. 122, Springer-Verlag, Berlin, 2006, Theory of $C\sp *$-algebras and von
  Neumann algebras, Operator Algebras and Non-commutative Geometry, III.

\bibitem{Bla:Cardinal}
A.~Blass, \emph{Combinatorial cardinal characteristics of the continuum},
  Handbook of Set Theory (M.~Foreman and A.~Kanamori, eds.), 2009.

\bibitem{BrOz:C*}
N.~Brown and N.~Ozawa, \emph{{$C\sp *$}-algebras and finite-dimensional
  approximations}, Graduate Studies in Mathematics, vol.~88, American
  Mathematical Society, Providence, RI, 2008.

\bibitem{DiLa:Deux}
J.~Dixmier and E.~C. Lance, \emph{Deux nouveaux facteurs de type {${\rm
  II}_{1}$}}, Invent. Math. \textbf{7} (1969), 226--234.

\bibitem{EllTo:Regularity}
G.A. Elliott and A.S. Toms, \emph{Regularity properties in the classification
  program for separable amenable {$C\sp *$}-algebras}, Bull. Amer. Math. Soc.
  \textbf{45} (2008), no.~2, 229--245.

\bibitem{Fa:Graphs}
I.~Farah, \emph{Graphs and {CCR} algebras}, Indiana Univ. Math. Journal
  \textbf{59} (2010), 1041Ð1056.

\bibitem{FaHaSh:Model3}
I.~Farah, B.~Hart, and D.~Sherman, \emph{Model theory of operator algebras
  {III}: Elementary equivalence and {II}$_1$ factors}, preprint,
  arXiv:1111.0998, 2011.

\bibitem{FaKa:NonseparableII}
I.~Farah and T.~Katsura, \emph{Nonseparable {UHF} algebras {II}:
  classification}, preprint, to appear.

\bibitem{FaKa:Nonseparable}
I.~Farah and T.~Katsura, \emph{Nonseparable {UHF} algebras {I}: {D}ixmier's
  problem}, Adv. Math. \textbf{225} (2010), no.~3, 1399--1430.

\bibitem{Fr:MA}
D.H. Fremlin, \emph{Consequences of martin's axiom}, Cambridge University
  Press, 1984.

\bibitem{HirKirWhi}
Ilan Hirshberg, Eberhard Kirchberg, and Stuart White, \emph{Decomposable
  approximations of nuclear {$C^*$}-algebras}, Adv. Math. \textbf{230} (2012),
  no.~3, 1029--1039. \MR{2921170}

\bibitem{Kat:Non-separable}
T.~Katsura, \emph{Non-separable {AF}-algebras}, Operator Algebras: The Abel
  Symposium 2004, Abel Symp., vol.~1, Springer, Berlin, 2006, pp.~165--173.

\bibitem{Tak:TheoryIII}
M.~Takesaki, \emph{Theory of operator algebras. {III}}, Encyclopaedia of
  Mathematical Sciences, vol. 127, Springer-Verlag, Berlin, 2003, Operator
  Algebras and Non-commutative Geometry, 8.

\bibitem{ToWin:Elliott}
Andrew~S. Toms and Wilhelm Winter, \emph{The {E}lliott conjecture for
  {V}illadsen algebras of the first type}, J. Funct. Anal. \textbf{256} (2009),
  no.~5, 1311--1340.

\bibitem{Wid:Nonisomorphic}
H.~Widom, \emph{Nonisomorphic approximately finite factors}, Proc. Amer. Math.
  Soc. \textbf{8} (1957), 537--540. \MR{0086275 (19,155b)}

\bibitem{Win:Nuclear}
W.~Winter, \emph{Nuclear dimension and {$\mathcal Z$}-stability of pure
  {C*}-algebras}, Invent. Math. \textbf{187} (2012), 259--342.

\bibitem{Win:Ten}
\bysame, \emph{Ten lectures on topological and algebraic regularity properties
  of nuclear {C*}-algebras}, CBMS conference notes, to appear.

\bibitem{WinterZacharias:NucDim}
Wilhelm Winter and Joachim Zacharias, \emph{The nuclear dimension of
  {$C^\ast$}-algebras}, Adv. Math. \textbf{224} (2010), no.~2, 461--498.
  \MR{2609012 (2011e:46095)}

\end{thebibliography}
\end{document}